\documentclass[a4paper,10pt]{article}
\usepackage{amssymb}
\usepackage{mathrsfs}
\usepackage{amsfonts}
\usepackage{amsmath}
\usepackage{latexsym}
\usepackage{mathbbol}
\usepackage{amsthm}
\usepackage[sans]{dsfont}
\usepackage[mathscr]{euscript}
\usepackage[nospace]{overcite}
\usepackage[nomargin,index]{fixme}
\usepackage{graphicx}

 \newtheorem{theorem}{Theorem}
 \newtheorem{corollary}{Corollary}
 
 \newtheorem{proposition}{Proposition}

\renewcommand{\d}{\mathrm{d}}

\title{Stopping Time and Control for a Type of\\Impulsive Stochastic Differential Equation}
\author{Ricardo Castro Santis\\Departamento de Matem\'atica, Universidad del B\'io-B\'io}
\date{}
\begin{document}

\maketitle

\begin{abstract}
The main objective of this paper is the construction of the solution of an impulsive stochastic differential equation, subject to control conditions in the pulse times, and to give sufficient conditions for the pulse times to be random variables with finite expectations. Such equations are useful in modeling diverse phenomena, such as biological control and pressure regulating mechanisms. The article ends with an application to fisheries.
\end{abstract}

\section{Introduction}

Stochastic differential equations, SDEs, have been an important tool in the description of many phenomena in different areas of knowledge. The main results as to their existence, uniqueness, and qualitative properties are found in the classical literature, such as [\citen{Baldi-2000,Oksendal-2003,Karatzas-Shere-1991}] and in recent years, a line of research has been exploring problems related to stochastic differential equations with pulses (see [\citen{Higham-2001,Sakthivel-Luo-2009,WU-SUN-2006}] ).\\

This paper focuses on determining the finiteness of the expectation of the pulse times  that occur when the solution process of a SDE reaches a control function $s(t)$, producing a pulse that sends the process to a second control function $q(t)$.\\

To this end, consider a one-dimensional Brownian motion $\ B_t\ $ defined over a stochastic basis $(\Omega,\mathscr{F},(\mathscr{F})_t,\mathbb{P})$, two real and continuous functions $q$ and $s$, and a stochastic process $X(\omega,t):(\Omega,\mathbb{R}_+)\to\mathbb{R}$ solution of a SDE of the following type:
$$\d X(t)=f(X(t))\d t+\sigma X(t)\d B_t,$$
with $X(0)=q(0)${\ } and{\ } $\tau=\inf\{t>0; X(t)=s(t)\}$. In this time, the process will be sent to the value $X(\tau^+)=q(\tau)$.\\

The functions $q$ and $s$ represent the curves of the lower and upper control functions, respectively.

\section{The Model}

Let $q$ and $s$ be two continuous and positive functions such that $q(t)<s(t)$ for all $t\ge0$, let $f\in\mathrm{C}^1(\mathbb{R})$, and let $B_t$ be a standard Brownian motion. Then the model can be written as

\begin{equation}\label{eq:main}
\left\{\begin{array}{rclcl}
        \d X(t)&=&f(X(t))\d t+\sigma X(t)\d B_t&{}& t\in]\tau_k,\tau_{k+1}]\\
        X(\tau_k^+)&=&q(\tau_k)\\
        \tau_{k}&=&\inf\Big\{t>\tau_{k-1};\quad X(t)=s(t)\Big\}&{},&\quad\tau_0=0
       \end{array}
\right.
\end{equation}

\vspace{0.5cm}
\begin{center}
      \includegraphics[width=9cm]{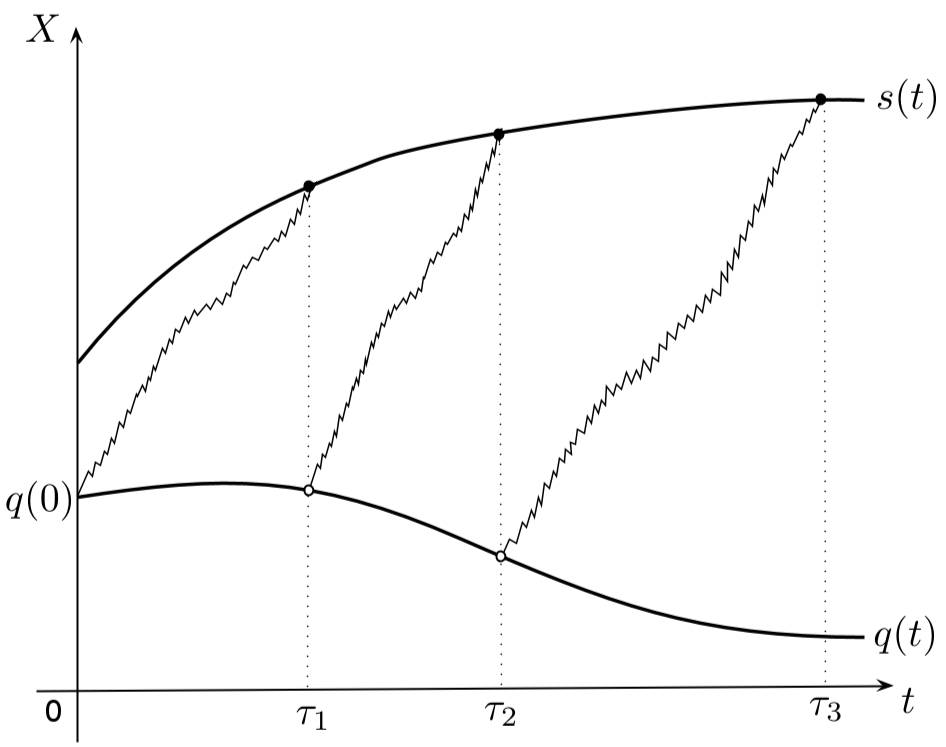}
      
      \textbf{\begin{footnotesize}Figure 1\end{footnotesize}}
\end{center}

\vspace{0.3cm}
\noindent This figure shows a possible path of evolution of the system proposed. The times $\tau_1, \tau_2$ and $\tau_3$ correspond to the first pulse times of the system.\\

\noindent The main goal is to prove that it is possible to obtain a behaviour as in Figure 1, that a local solution of the equation $\d X(t)=f(X(t))\d t+\sigma X(t)\d B_t$ exists on the stochastic interval $]\tau_{k-1},\tau_k]$, and that the pulse times $\tau_k$ are random variables with finite expectations.

\section{HYPOTHESES}

 \begin{enumerate}
  \item[\textbf{(A)}] Suppose that $f\in\mathrm{C}^1(\mathbb{R_+})$ and that there exist positive constants $\alpha, \beta$ and $S$ such that $0\le\alpha<\beta$ and
    \[
   \alpha x\le f(x)\le\beta x,\quad x\in[0,S].
  \]
  \item[\textbf{(B)}] Assume that the positive constants $\alpha$ and $\sigma$ satisfy $\alpha>\dfrac{\sigma^2}{2}$.\\
  
  \item[\textbf{(C)}] Assume that $q$ and $s$ are two continuous functions such that $0<q(t)<s(t)<S$ for all $t\ge0$.
  \end{enumerate}

\vspace{0.5cm}
 \begin{minipage}{6.3cm}
      \includegraphics[width=6cm]{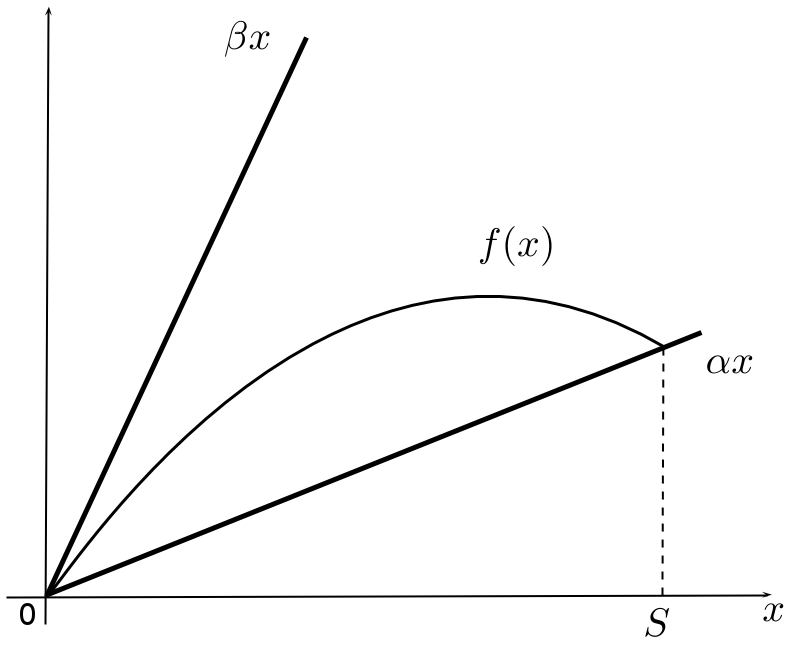}
      \begin{center}\textbf{\begin{footnotesize}Figure 2\end{footnotesize}}\end{center}
     \end{minipage}
\begin{minipage}{5.7cm}
 \noindent \textbf{NOTE:} Hypothesis (A) implies that
 
 \begin{itemize}
  \item $f(0)=0$,\quad if $\left.\dfrac{df}{dx}\right|_{x=0}$ exists, then\\

  \vspace{0.2cm}
  \item $\alpha\le f'(0)\le\beta$.\\
 
 \end{itemize}
\end{minipage}

\vspace{0.3cm}
 Figure 2 shows the relation of Hypothesis (A).
\vspace{0.5cm}
\begin{proposition}\label{prop:exist-1}
Suppose that $f$ satisfies Hypothesis \emph{(A)} and that $x_0\in]0,S[$ and $\sigma$ are positive constants. Then there exists a unique stochastic process $X(t)$ such that

\begin{equation}\label{eq:main-1}
\left\{\begin{array}{rcl}
    \d X(t)&=&f(X(t))\d t+\sigma X(t)\d B(t)\\
    {}\\
    X(0)&=&x_0
          \end{array}\right.
\end{equation}
on the stochastic interval $[0,\tau]$ where $\displaystyle\tau=\inf\left\{t>0;\quad 0\le X(t)\le S\right\}$.
\end{proposition}

\begin{proof}
 Hypothesis (A) implies that $f$ is a Lipschitz and sublinear function on the stochastic interval $[0,\tau]$, where $\displaystyle\tau=\inf\left\{t>0;\quad 0< X(t)\le S\right\}$. Therefore the application of the classical theorem of the existence of a solution to a stochastic differential equation with local Lipschitz coefficients (see Baldi\cite{Baldi-2000} or \O ksendal \cite{Oksendal-2003}) implies that \eqref{eq:main-1} has a unique solution process on the interval $[0,\tau]$.
\end{proof}

\noindent\textbf{NOTE:} Under Hypothesis (A), the process $X(t)$ is a stopping time with continuous paths, therefore
\[
 \displaystyle\tau=\inf\left\{t>0;\quad 0< X(t)\le S\right\}=\inf\left\{t>0;\quad 0< X(t)= S\right\}.
\]

\begin{proposition}\label{prop:order-solution}
 Suppose that $f$ and $\sigma$ satisfy Hypothesis \emph{(A)}. Then the stochastic process $X(t)$ that is the solution of \eqref{eq:main-1} satisfies 
 
 \begin{equation}\label{eq:X-ineq}
X_\alpha\le X(t)\le X_\beta(t),
 \end{equation}
where
$X_\alpha(t)=x_0e^{\left(\alpha-\frac{\sigma^2}{2}\right)t+\sigma B(t)}$\quad and\quad $X_\beta(t)=x_0e^{\left(\beta-\frac{\sigma^2}{2}\right)t+\sigma B(t)}$.
\end{proposition}

\begin{proof}
Note that if $f(x)\le g(x)$ with $x\in I$, then $x_1(t)\le x_2(t)$, where $x_1$ and $x_2$ are the solutions of the equations 
$x'=f(x)$ and $x'=g(x)$ with $x_1(0)=x_2(0)=x_0$.\\

The fact that $ \alpha x\le f(x)\le\beta x,\quad x\in[0,S]$ implies that the processes $X_\alpha(t)$, $X(t)$ and $X_\beta(t)$ are solutions to the equations

\[\begin{array}{crcl}
  (i)& \d X(t)&=&\alpha X(t)\d t+\sigma X(t)\d B(t),\\
  (ii)& \d X(t)&=&f(X(t))\d t+\sigma X(t)\d B(t)\quad \mbox{and}\\
  (iii)& \d X(t)&=&\beta X(t)\d t+\sigma X(t)\d B(t)
  \end{array}
\]
respectively, with $X_\alpha(0)=X(0)=X_\beta(0)=x_0\in]0,S[$\\

\noindent satisfying $X_\alpha(t)\le X(t)\le X_\beta(t)$ on the stochastic interval  $\displaystyle\tau=\inf\left\{t>0;\quad 0\le X(t)\le S\right\}$.\\

\noindent Furthermore, the explicit solutions of (\textit{i}) and (\textit{iii}) correspond to the geometric Brownian motions $X_\alpha=x_0e^{\left(\alpha-\frac{\sigma^2}{2}\right)t+\sigma B(t)}$ \quad and\quad $X_\beta=x_0e^{\left(\beta-\frac{\sigma^2}{2}\right)t+\sigma B(t)}$, respectively (see {\O}ksendal\cite{Oksendal-2003}, p. 63). So the result is obtained.\\
\end{proof}

\noindent \textbf{Note:} Proposition \ref{prop:order-solution} implies the positivity of the solutions. If Hypothesis (B) is not satisfied, and in particular if $\alpha<\dfrac{\sigma^2}{2}$, then
\[
 0<X(t)\le X_\beta(t) \longrightarrow 0\qquad \mbox{where } t\to\infty\qquad\mbox{a.s.}\\ 
\]

\begin{theorem}\label{theo:main}
If $f$ and $\sigma$ satisfy Hypotheses \emph{(A)} and \emph{(B)}, then the stopping time $\displaystyle\tau=\inf\left\{t>0;\quad 0< X(t)= S\right\}$ is a random variable with finite expectation.
\end{theorem}

\begin{proof}
From Proposition \ref{prop:order-solution} we have $X_\alpha(t)\le X(t),\quad\forall t\ge0$, therefore $\tau\le\tau_\alpha$, where $\displaystyle\tau=\inf\left\{t>0;\quad 0< X(t)= S\right\}$\quad and\quad  $\displaystyle\tau_\alpha=\inf\left\{t>0;\quad 0< X_\alpha(t)= S\right\}$.\\

From the equation  $X_\alpha(t)=S$ we have that

$$\left(\alpha-\dfrac{\sigma^2}{2}\right)t+B(t)=\ln\left(\dfrac{S}{x_0}\right),$$
and therefore
 \begin{equation}\label{eq:Exp-tau-alpha}
    \mathbb{E}(\tau_\alpha)=\frac{\ln\left(\dfrac{S}{x_0}\right)}{\alpha-\dfrac{\sigma^2}{2}}
 \end{equation}
By hypothesis, $x_0<S$ and $\alpha>\dfrac{\sigma^2}{2}$,\quad and so $0<\mathbb{E}(\tau_\alpha)<\infty$, completing the theorem. \\
 
\end{proof}

\noindent\textbf{NOTE:} Similar arguments can be used to deduce that $\mathbb{E}(\tau_\beta)\le\mathbb{E}(\tau)$, where $\displaystyle\tau_\beta=\inf\left\{t>0;\quad 0< X_\beta(t)= S\right\}$, therefore the expectation for $\tau$ is bounded by

 \begin{equation}\label{eq:Exp-tau-alpha-beta}
    \mathbb{E}(\tau_\beta)=\frac{\ln\left(\dfrac{S}{x_0}\right)}{\beta-\dfrac{\sigma^2}{2}}\le\mathbb{E}(\tau)\le\frac{\ln\left(\dfrac{S}{x_0}\right)}{\alpha-\dfrac{\sigma^2}{2}}=\mathbb{E}(\tau_\alpha).
 \end{equation}

 \noindent The following figure shows the relation between $\tau_\alpha,\ \tau$ and $\tau_\beta$.

\vspace{0.5cm}
\begin{center}
 \includegraphics[width=9cm]{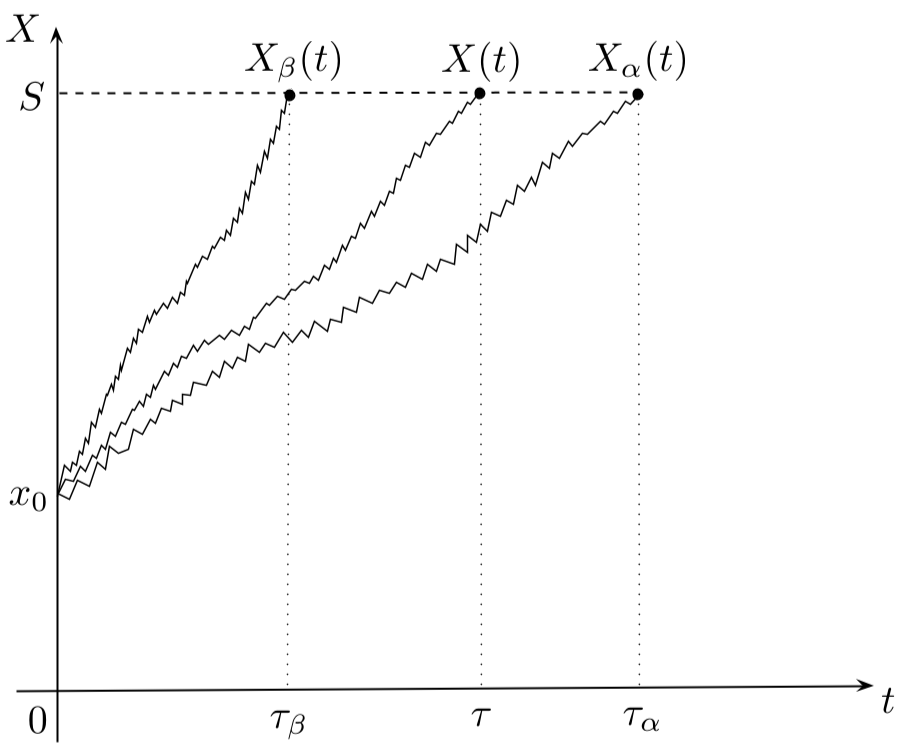}
 
 \textbf{\begin{footnotesize}Figure 3\end{footnotesize}}
\end{center}

 A study of the first exit time of diffusion, can be seen in [\citen{Platen-1985}], but with assumptions and different approach to this work
 
\vspace{0.5cm}
\section{The pulses and the control}

The interest of this paper is to bring the process from a state on the curve $q(t)$ to a state on the curve $s(t)$ and return it, instantaneously, to a state in $q(t)$, as is shown in Figure 1. Therefore consider the equations before the first pulse,

\begin{equation}\label{eq:stepp-1}
\left\{\begin{array}{rclcl}
        \d X(t)&=&f(X(t))\d t+\sigma X(t)\d B_t&{}&{}\\
        X(0^+)&=&q(0),
       \end{array}
\right.
\end{equation}
and the random variable $\tau_{1}=\inf\Big\{t>0;\quad X(t)=s(t)\Big\}$.

\begin{corollary}\label{coro:stepp-1}
 Under Hypotheses \emph{(A)}, \emph{(B)} and \emph{(C)}, Equation \eqref{coro:stepp-1} has a unique solution on $0<X(t)<S$, and $\tau_1$ is a random variable with finite expectation.
\end{corollary}

\noindent\textbf{NOTE:} Due to the continuity of $X(t)$ and $s(t)$, the random time $\tau_1$ is well defined.

\begin{proof}
 Due to Hypothesis (C), $X(0)>0$ and $X(0)<s(t)<S,\ \forall t$, it  is possible to use Proposition \ref{prop:exist-1}. Therefore, there  exists a unique stochastic process $X_1(t)$ solution of Equation \eqref{eq:stepp-1} on the interval $]0,\tau_1]$. The fact that $s(t)<S$ implies that $\tau_1<\tau,\quad  \forall\omega\in\Omega$, where $\tau$ is the stopping time defined in Proposition \ref{prop:exist-1} taking $x_0=q(0)$. Therefore $\mathbb{E}[\tau_1]<\mathbb{E}[\tau]<\infty$.
\end{proof}

\begin{corollary}\label{coro:shift-T}
 Let $T>0$. The equation temporally displaced by $T$,
 
\begin{equation}\label{eq:shift-T}
\left\{\begin{array}{rclcl}
        \d X(t)&=&f(X(t))\d t+\sigma X(t)\d B_t&{}\\
        X(T^+)&=&q(T),\\
       \end{array}
\right.
\end{equation}
has a unique solution on $0<X(t)<S$, and the random variable $\tau_2$ defined by\\ $\tau_{2}=\inf\Big\{t>T;\quad X(t)=s(t)\Big\}$ has a finite expectation.
\end{corollary}

\begin{proof}
The change of variables $u=t-T$ transforms Equation \eqref{eq:shift-T} into

\begin{equation}\label{eq:shift-0}
 \left\{\begin{array}{rclc}
        \d X(u)&=&f(X(u))\d u+\sigma X(u)\d B^{(1)}_u&{}\\
        X(0^+)&=&q(T),
       \end{array}
\right.
\end{equation}
where $B^{(1)}_u=B_u-B_T$ is a new standard Brownian motion and $\tau_2=T+\tau$,\ with $\tau$ the stopping time defined in Proposition \ref{prop:exist-1} taking $x_0=q(T)$.\\

\noindent Therefore, it is possible to apply Corollary \ref{coro:stepp-1} with initial condition $X(0^+)=q(T)$ and Brownian motion $B^{(1)}$, thus  obtaining a stochastic process $X(u)$ solution of Equations \eqref{eq:shift-0} with $u\in]0,\tau_1]$\  that is equivalent to the solution $X(t)$ of Equations \eqref{eq:shift-T} with $t\in]T,\tau_2]$. Moreover $\mathbb{E}[\tau_2]=T+\mathbb{E}[\tau]<\infty$.
\end{proof}

\noindent Note that between $\mathbb{E}[\tau_2]$ and $\mathbb{E}[\tau]$ there is no relationship of order.\\

\begin{theorem}[Main results] Under Hypotheses \emph{(A)}, \emph{(B)} and \emph{(C)}, there exists a stochastic process $X(t)$ that is the unique solution, in the a.s. sense, of the system \eqref{eq:main}, and the pulse times $\tau_k$ are random variables with finite expectations.
\end{theorem}

\begin{proof}
\noindent Using Corollary \ref{coro:stepp-1} and Corollary \ref{coro:shift-T}, it is possible to obtain $X_1(t)$, a solution of Equation \eqref{eq:stepp-1} on the interval $]0,\tau_1]$, and $X_2(t)$, a solution of Equation \eqref{eq:shift-T} with initial condition $X(\tau_1^+)=q(\tau_1)$ on the interval $]\tau_1,\tau_2]$.\\

\noindent Then the process $X(t)=X_1(t)\mathbb{1}_{]0\tau_1]}(t)$ is a solution of the system \eqref{eq:main} on the stochastic interval $]0\tau_1]$ and $X(t):=X_2(t)\mathbb{1}_{]\tau_1,\tau_2]}(t)$ is a solution of the system \eqref{eq:main} on the stochastic interval $]\tau_1,\tau_2]$. Moreover, $X(t)=X_1(t)\mathbb{1}_{]0\tau_1]}(t)+X_2(t)\mathbb{1}_{]\tau_1,\tau_2]}(t)$ is a solution on the interval $]0,\tau_2]$. Therefore, using the same argument recursively, we find

\begin{equation}\label{eq:solution-main}
X(t):=\sum_{k=1}^\infty X_k(t)\mathbb{1}_{]\tau_{k-1},\tau_k]}(t),
\end{equation}
a global solution of the system \eqref{eq:main}. Its uniqueness a.s. is due to the uniqueness on each interval.\\

\noindent The expectation of the stopping time $\tau_k=\inf\Big\{t>\tau_{k-1};\quad X(t)=s(t)\Big\}$ is given by the following recurrence equation.\\
\begin{equation}\label{eq:stopping-time}
\mathbb{E}[\tau_k]=\mathbb{E}[\tau_{k-1}]+\mathbb{E}[\Delta\tau_k] 
\end{equation}
where $\Delta\tau_k$ is the stopping time defined in Proposition \ref{prop:exist-1} taking $x_0=q(\tau_{k-1})$.
\end{proof}

The random variable $\Delta\tau_k=\tau_k-\tau_{k-1}$ corresponds to the time between two successive pulses, called the $k$th \emph{timeout}.

\section{Asymptotic behaviour}

In order to have more precise information about the asymptotic behaviour of the process, it is necessary to impose additional conditions on the functions $q$ and $s$. More precisely, if $q$ is increasing and $s$ is decreasing, or vice versa, then it is possible to improve the estimates for the timeout $\Delta\tau_k$.

\begin{proposition}\label{prop:q-in s-de}
Assume that Hypotheses \emph{(A)}, \emph{(B)} and \emph{(C)} are satisfied and  further, let $s$ be a decreasing function and $q$ be an increasing function. Then $\mathbb{E}[\Delta\tau_k]$ is a decreasing sequence.
\end{proposition}

\begin{proof}
 Consider the auxiliary times $\tau_q$ and $\tau_s$ defined by
  \begin{equation}\label{eq:auxiliary-time}
  \tau_q=\inf\Big\{t>\tau_{k-1};\ X_{k}(t)=q(\tau_k)\Big\}\quad\mbox{y}\ \tau_s=\inf\Big\{t>\tau_q;\quad X_{k}(t)=s(\tau_{k+1})\Big\}.
 \end{equation}
\noindent\textbf{Note:} $(X_k,\tau_k)$ is the solution to the system \eqref{eq:main}, and the fact that $q$ is increasing and $s$ decreasing implies $\tau_{k-1}<\tau_q<\tau_s<\tau_k$.

\vspace{0.2cm}
\noindent The following graph describes the situation:

\begin{center}
 \includegraphics[width=8cm]{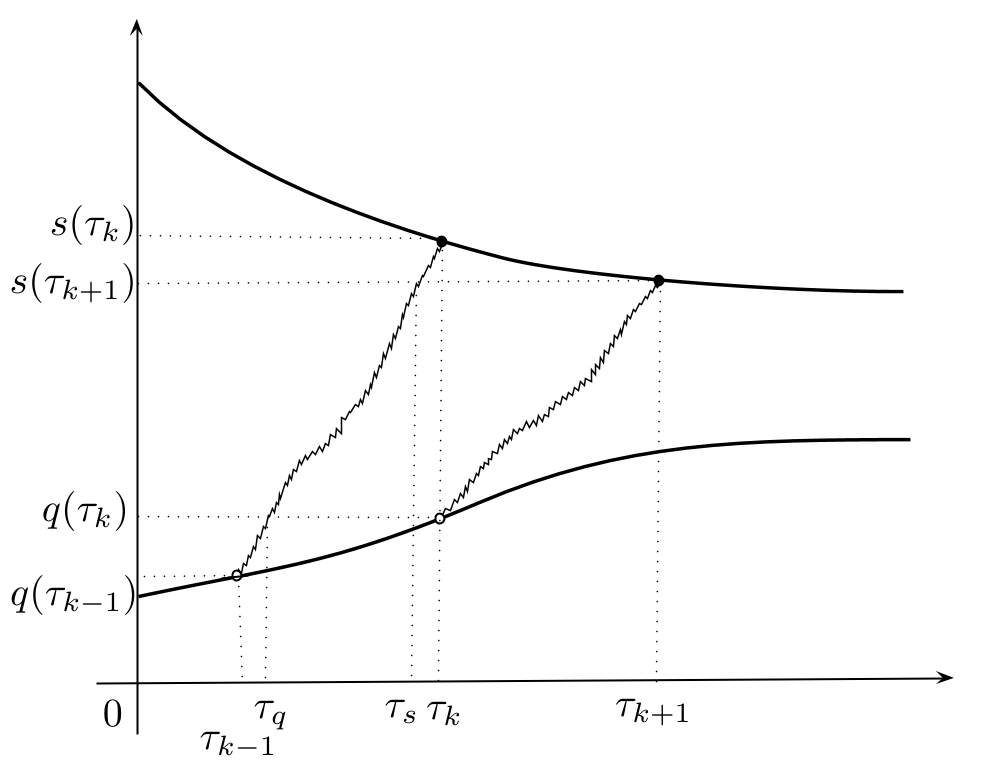}

 \textbf{\begin{footnotesize}Figure 4\end{footnotesize}}
\end{center}

\noindent\ Using the temporal order, it is possible to decompose $\Delta\tau_k$ as follows. 

\begin{equation}\label{eq:decompose timeout-1}
\Delta\tau_k=(\tau_q-\tau_{k-1})+(\tau_s-\tau_q)+(\tau_k-\tau_s)
\end{equation}
However, the random variable $\tau_s-\tau_q$ corresponds to the timeout of Equation \eqref{eq:shift-T} with initial time $T=\tau_q$, initial value $q(\tau_k)$, and final value $s(\tau_{k+1})$, and is therefore a random variable distributed identically to the timeout $\Delta\tau_{k+1}$. Now taking expectations in Equation \eqref{eq:decompose timeout-1}, we have

\[\begin{array}{rcl}
\mathbb{E}[\Delta\tau_k]&=&\mathbb{E}[\tau_q-\tau_{k-1}]+\mathbb{E}[\tau_s-\tau_q]+\mathbb{E}[\tau_k-\tau_s]\\
{}&=&\mathbb{E}[\tau_q-\tau_{k-1}]+\mathbb{E}[\Delta\tau_{k+1}]+\mathbb{E}[\tau_k-\tau_s].
\end{array}\]

\noindent Therefore\quad $\mathbb{E}[\Delta\tau_k]-\mathbb{E}[\Delta\tau_{k+1}]= \mathbb{E}[\tau_q-\tau_{k-1}]+\mathbb{E}[\tau_k-\tau_s]>0$\quad , completing the proof.
\end{proof}

\begin{proposition}\label{prop:q-de s-in}
Assume that Hypotheses \emph{(A)}, \emph{(B)} and \emph{(C)} are satisfied, and that $s$ is an increasing function and $q$ is a decreasing function. Then $\mathbb{E}[\Delta\tau_k]$ is an increasing sequence.
\end{proposition}

\begin{proof}
 The proof is similar to that of the previous proposition. In this case, the temporal order is $\tau_{k-1}<\tau_k<\tau_q<\tau_s<\tau_{k+1}$, with $\tau_q$ and $\tau_s$ defined analogously to Equation \eqref{eq:auxiliary-time}. The $k+1$th timeout is decomposed as follows.
\begin{equation}\label{eq:decompose timeout-2}
\Delta\tau_{k+1}=(\tau_{q}-\tau_{k})+(\tau_s-\tau_q)+(\tau_{k+1}-\tau_s)
\end{equation}
Now, the random variable $\tau_s-\tau_q$ corresponds to the timeout of Equation \eqref{eq:shift-T} with initial time $T=\tau_q$, initial value $q(\tau_{k-1})$, and final value $s(\tau_{k})$. It is therefore a random variable identically distributed to the timeout $\Delta\tau_{k}$. Now taking expectations in Equation \eqref{eq:decompose timeout-1}, we have

\[\begin{array}{rcl}
\mathbb{E}[\Delta\tau_{k+1}]&=&\mathbb{E}[\tau_q-\tau_{k}]+\mathbb{E}[\tau_s-\tau_q]+\mathbb{E}[\tau_{k+1}-\tau_s]\\
{}&=&\mathbb{E}[\tau_q-\tau_{k}]+\mathbb{E}[\Delta\tau_{k}]+\mathbb{E}[\tau_{k+1}-\tau_s].
\end{array}\]

\noindent Therefore, $\mathbb{E}[\Delta\tau_{k+1}]-\mathbb{E}[\Delta\tau_{k}]= \mathbb{E}[\tau_q-\tau_{k}]+\mathbb{E}[\tau_{k+1}-\tau_s]>0$, which completes the proof.
\end{proof}

Note that if $q$ and $s$ are monotone, then under Hypothesis \textbf{C}, $\displaystyle\lim_{t\to\infty}q(t)=Q\in]0,S[$\quad and\quad $\displaystyle\lim_{t\to\infty}s(t)=\tilde{S}\in]0,S[$\quad with $0<Q\le\tilde{S}<S$. This allows the following theorem.

\begin{theorem}\label{theo:Delta convergence}
If Hypotheses \emph{\textbf{A}}, \emph{\textbf{B}} and \emph{\textbf{C}} are satisfied, if $q$ is increasing and $s$ is decreasing, or if $q$ is decreasing and $s$ is increasing, then the expectation of the timeout $\Delta\tau_k$ converges on the interval $[a_\beta,a_\alpha]$, with 
\[
a_\beta=\frac{1}{\beta-\dfrac{\sigma^2}{2}}\ln\left(\dfrac{\tilde{S}}{Q}\right)\quad\mbox{\emph{and}}\quad a_\alpha=\frac{1}{\alpha-\dfrac{\sigma^2}{2}}\ln\left(\dfrac{\tilde{S}}{Q}\right).
\]
\end{theorem}

\begin{proof} {\ }

\begin{enumerate}
 \item If $s$ is decreased and $q$ is increased then, by Proposition \ref{prop:q-in s-de}, the sequence of the expectations of a timeout, $\mathbb{E}[\Delta\tau_k]$, is a decreasing sequence. In the expression for the expectation of $\tau_\alpha$ in Equation \eqref{eq:Exp-tau-alpha}, the initial and final values can be replaced and updated at time $k$ by $x_0\to q(\tau_{k-1})$ and $S\to q(\tau_{k-1})$. Therefore the expectation of the $k$th timeout $\Delta\tau_k$ is bounded by
 
 \begin{equation}\label{eq:alphak-Exp-betak}
\frac{1}{\beta-\dfrac{\sigma^2}{2}}\ln\left(\dfrac{s(\tau_{k-1}}{q(\tau_{k-1})}\right)<\mathbb{E}[\Delta\tau_k]< \frac{1}{\alpha-\dfrac{\sigma^2}{2}}\ln\left(\dfrac{s(\tau_{k-1}}{q(\tau_{k-1})}\right),
 \end{equation}
\noindent and taking limits,

\begin{equation}\label{eq:alpha-Exp-beta}
\frac{1}{\beta-\dfrac{\sigma^2}{2}}\ln\left(\dfrac{\tilde{S}}{Q}\right)\le\lim_{k\to\infty}\mathbb{E}[\Delta\tau_k]\le \frac{1}{\alpha-\dfrac{\sigma^2}{2}}\ln\left(\dfrac{\tilde{S}}{Q}\right),
\end{equation}
 and so $\mathbb{E}[\Delta\tau_k]$ is a decreasing and bounded sequence, therefore it has a limit in the interval $[a_\beta,a_\alpha]$.
 
 \vspace{0.3cm}
 \item The case where $s$ is increasing and $q$ is decreasing is analogous to the previous case, but then $\mathbb{E}[\Delta\tau_k]$ is an increasing sequence.
\end{enumerate}
\end{proof}

\section{Applications}

\begin{enumerate}
 \item \textbf{Fishery with fixed quota:} We consider a fishery resource with long periods of closure. When the amount of resources reaches a certain value $S$, fishing is allowed for a very short period of time with a total fishing quota $C$, and then the closure of the resources begins again. The evolution of the resources follows a logistic growth law in a random environment. This can be modeled by the following impulsive system.
 
 \begin{equation}\label{eq:example-1}
  \left\{\begin{array}{rclcl}
        \d X(t)&=&r\left(1-\dfrac{X(t)}{K}\right)X(t)\d t+\sigma X(t)\d B_t&{}& t\in]\tau_k,\tau_{k+1}]\\
        {}\\
        X(\tau_k^+)&=&S-C\\
        {}\\
        \tau_{k}&=&\inf\Big\{t>\tau_{k-1};\quad X(t)=S\Big\},&{}&\tau_0=0,
        \end{array}\right.
 \end{equation}
 
 \noindent where $r$ is the intrinsic growth rate and $K$ the capacity of the environment. It must be verified that Hypotheses (A), (B) and (C) are met.\\
 
The population dynamics in a random environment has been extensively studied by many authors; a complete compendium of the main results can be found in [\citen{Lande-Enge-2003}]. Applications to fisheries have been studied in [\citen{Braumann01}]. The regulation of fishing and fishing quotas in a random environment have been addressed in [\citen{Braumann04}]. Impulsive models of a fishery have been treated in [\citen{Cordova02}] and [\citen{Zhao}].
  
 \begin{itemize}
  \item Hypothesis (A): In this case, $f(x)=r\left(1-\dfrac{x}{K}\right)x$, the constant $S$ may be chosen as any positive value less that $K$. Thus, the problem makes sense.
  
  $$\mbox{The constant}\quad\beta=f'(0)=r\quad\mbox{and the constant}\quad\alpha=\dfrac{f(S)}{S}=\dfrac{r}{K}(K-S).$$
  
  \item Hypothesis (B): The constant $\sigma$ can be anything such that $\sigma^2<2\alpha$.\\
  
  \item Hypothesis (C): In this case, $s(t)\equiv S$ and $q(t)\equiv S-C$.
 \end{itemize}
Due to the fact that $q$ and $s$ are constants, the timeouts $\Delta_k$ are independent and  identically distributed random variables, therefore $\mathbb{E}[\Delta_k]=E$ is constant on the interval $[a_\beta,a_\alpha]$, with
\[
 a_\beta=\frac{2}{r-\sigma^2}\ln\left(\frac{S}{S-C}\right)\quad\mbox{and}\quad a_\alpha= \frac{2K}{2r(K-S)-K\sigma^2}\ln\left(\frac{S}{S-C}\right).
\]

\vspace{0.2cm}
\item \textbf{Fishery with programming of total closure:} The context of this application is a resource that will have a total prohibition on its being fished, but during a period there will be allowed a quota, progressively smaller, only for short periods of fishing. As in the previous case, the fishing is permitted only when the resources reach the value $S$, but the quota is a fraction of the amount of the previous fishing, depending on the timeout:  this is

\begin{equation*}
\frac{C(\tau_k)}{C(\tau_{k-1})}=\gamma^{\frac{\Delta_k}{T}}
\end{equation*}

\noindent where $\gamma\in]0,1[$ is the fraction of fishing and $T$ is a temporal scale factor.

Note that if $\Delta_k$ is constant, then $C(\tau_k)$ is a geometric progression. A reasonable choice for $C$ is $C(\tau_k)=S\gamma^{1+\frac{\tau_k}{T}}$, as in the previous example, $q(t)=S-C$, and  therefore  $q(t)=S\left(1-\gamma^{1+\frac{t}{T}}\right)$ and the model can be written as

 \begin{equation}\label{eq:example-2}
  \left\{\begin{array}{rclcl}
        \d X(t)&=&r\left(1-\dfrac{X(t)}{K}\right)X(t)\d t+\sigma X(t)\d B_t&{}& t\in]\tau_k,\tau_{k+1}]\\
        {}\\
        X(\tau_k^+)&=&S\left(1-\gamma^{1+\frac{\tau_{k-1}}{T}}\right)&\\
        {}\\
        \tau_{k}&=&\inf\Big\{t>\tau_{k-1};\quad X(t)=S\Big\},&{}&\tau_0=0.
        \end{array}\right.
 \end{equation}
 
As $\gamma\in]0,1[$, then $q(0)=S(1-\gamma)\in\ ]0,S[$\quad and\quad $\displaystyle\lim_{k\to\infty} q(t)=S$.\\

As a consequence of Equation \eqref{eq:alpha-Exp-beta}, one obtains that $\displaystyle\lim_{k\to\infty}\mathbb{E}[\Delta \tau_k]=0$.

\end{enumerate}


\begin{thebibliography}{99}
\bibitem{Baldi-2000} BALDI, Paolo (1984). \textit{Equazioni differenziali stocastiche e applicazioni}. Pitagora.

\bibitem{Braumann01} BRAUMANN, C.A. (1999). Variable effort fishing models in random environments. {\it Math. Biosci.} 156, 1--19.

\bibitem{Braumann04} BRAUMANN, C.A. (2001). Constant effort and constant quota fishing policies with cut-offs in a random environment. {\it Nat. Resour. Model.} 14, 199--232.

\bibitem{Cordova02} CORDOVA-LEPE, F., GONZALEZ-OLIVARES, E., \& PINTO, M. (2011). Source--sink impulsive bioeconomic models: Seasonal closures with fixed length. {\it J. Differ. Equ. Appl.} 17, 721--735. 

\bibitem{Higham-2001} HIGHAM, Desmond J. (2001). An algorithmic introduction to numerical simulation of stochastic differential equations. {\it SIAM Review}, 43(3), 525--546.

\bibitem{Karatzas-Shere-1991}
KARATZAS I. and SHREVE S. (1991). {\it Brownian Motion and Stochastic Calculus}. 2nd edn, Springer-Verlag, Berlin.

\bibitem{Lande-Enge-2003} LANDE, Russell; ENGEN, Steinar; SAETHER, Bernt-Erik (2003). {\it Stochastic Population Dynamics in Ecology and Conservation}. Oxford University Press.

\bibitem{Oksendal-2003}
\O KSENDAL, Bernt (2003). \textit{Stochastic Differential Equations}. Springer-Verlag, Berlin.

\bibitem{Platen-1985}
PLATEN E. (1985), On first exit time of diffusions.,  \textit{Lecture Note on Control and Information Sciences}, 69, 192--195
\bibitem{Sakthivel-Luo-2009} SAKTHIVEL, R.; LUO, J. (2009), Asymptotic stability of nonlinear impulsive stochastic differential equations. {\it Statistics \& Probability Letters} 79(9), 1219--1223.

\bibitem{WU-SUN-2006} WU, Huijun; SUN, Jitao. (2006), $p$-Moment stability of stochastic differential equations with impulsive jump and Markovian switching. {\it Automatica} 42(10), 1753--1759.

\bibitem{Zhao}  ZHAO, T., \& TANG, S. (2011). Impulsive harvesting and by-catch mortality for the theta logistic model., {\it Appl. Math. Comput.}  217, 9412--9423.

\end{thebibliography}
\end{document}